\documentclass[12pt]{article}
\usepackage{amsmath,amssymb,amsthm, amsfonts}
\usepackage{hyperref}
\usepackage{graphicx}
\usepackage{url}
\usepackage{dsfont} 
\usepackage[dvipsnames]{xcolor}

\usepackage{tikz}
\usepackage{tkz-graph}
\tikzstyle{vertex}=[circle, draw, inner sep=3pt, minimum size=6pt]

\usepackage{color}
\definecolor{red}{rgb}{1,0,0}

\definecolor{blue}{rgb}{0,0,1}

\definecolor{green}{rgb}{0,.6,0}

\usepackage{float}

\usepackage{tikz}

\newtheorem{thm}{Theorem}[section]
\newtheorem{cor}[thm]{Corollary}
\newtheorem{lem}[thm]{Lemma}
\newtheorem{prop}[thm]{Proposition}

\newtheorem{prob}[thm]{Problem}
\newtheorem{quest}[thm]{Question}

\theoremstyle{definition}

\theoremstyle{definition}

\theoremstyle{definition}

\theoremstyle{definition}

\theoremstyle{definition}

\newcommand{\binf}{B^{[\infty]}}

\newcommand{\Z}{\operatorname{Z}}

\newcommand{\bit}{\begin{itemize}}
\newcommand{\eit}{\end{itemize}}
\newcommand{\ben}{\begin{enumerate}}
\newcommand{\een}{\end{enumerate}}
\newcommand{\beq}{\begin{equation}}
\newcommand{\eeq}{\end{equation}}
\newcommand{\bea}{\begin{eqnarray}} 
\newcommand{\eea}{\end{eqnarray}}
\newcommand{\bpf}{\begin{proof}}
\newcommand{\epf}{\end{proof}\ms}
\newcommand{\bmt}{\begin{bmatrix}}
\newcommand{\emt}{\end{bmatrix}}
\newcommand{\ms}{\medskip}

\newcommand{\noi}{\noindent}

\newcommand{\beqs}{\begin{equation*}} 
\newcommand{\eeqs}{\end{equation*}}
\newcommand{\beas}{\begin{eqnarray*}}
\newcommand{\eeas}{\end{eqnarray*}}

\newcommand{\sub}[1]{_{(#1)}}

\newcommand{\calf}{\mathcal{F}}

\newcommand{\fs}{\rightarrow}

\title{On leaky forcing and resilience}

\author{Joseph S. Alameda\thanks{Dept.~of Mathematics, Iowa State University, Ames, IA, USA (jalameda@iastate.edu)} \and 
J\"urgen Kritschgau\thanks{Dept.~of Mathematics, Iowa State University, Ames, IA, USA (jkritsch@iastate.edu) Research is supported by NSF grant DMS-1839918} \and Nathan Warnberg\thanks{Dept.~of Mathematics and Statistics, University of Wisconsin-La Crosse, La Crosse WI, USA}
\and  Michael Young\thanks{
 Dept.~of Mathematics, Iowa State University, Ames, IA 50011, USA (myoung@iastate.edu) Research is supported by NSF Grant DMS-1719841.} }

\begin{document}

\maketitle

\begin{abstract}
    A leak is a vertex that is not allowed to perform a force during the zero forcing process.
    Leaky forcing was recently introduced as a new variation of zero forcing in order to analyze how leaks in a network disrupt the zero forcing process. The $\ell$-leaky forcing number of a graph is the size of the smallest zero forcing set that can force a graph despite $\ell$ leaks. A graph $G$ is $\ell$-resilient if its zero forcing number is the same as its $\ell$-leaky forcing number. In this paper, we analyze $\ell$-leaky forcing and show that if an $(\ell-1)$-leaky forcing set $B$ is robust enough, then $B$ is an $\ell$-leaky forcing set. This provides the framework for characterizing $\ell$-leaky forcing sets. Furthermore, we consider structural implications of $\ell$-resilient graphs. We apply these results to bound the $\ell$-leaky forcing number of several graph families including trees, supertriangles, and  grid graphs. In particular, we resolve a question posed by Dillman and Kenter concerning the upper bound on the $1$-leaky forcing number of grid graphs.
\end{abstract}

\noi {\bf Keywords} zero forcing, leaky forcing, color change rule

\noi{\bf AMS subject classification} 05C57, 05C15, 05C50
\section{Introduction}

Zero forcing is a process by which blue vertices propagate through a simple graph.\footnote{We use standard graph theoretic notation as introduced in \cite{W}.} More formally, we start with an initial set of blue vertices in a graph that color (or force) other vertices blue. A blue vertex $v$ can color (or force) a white vertex $w$ blue if $w$ is the only white neighbor of $v$. This is called \emph{ the zero forcing color change rule.} A set of vertices $B$ is a \emph{zero forcing set of $G$}, if $B$ can force every vertex in $G$ by iteratively applying the zero forcing color change rule.  The fewest blue vertices needed to turn the entire graph blue is called the \emph{zero forcing number} of the graph. Zero forcing was introduced in \cite{AIM} to find upper bounds for the maximum nullity for the family of real symmetric matrices whose nonzero off-diagonal entries are described by a graph.  However, the connections go beyond linear algebra and graph theory because, more generally, zero forcing is a process that models knowledge or control of systems.  In physics, the zero forcing process was used to optimally control quantum systems  \cite{BG,S}. In \cite{BDHSY}, it was shown that if a set of vertices is a zero forcing set, then the associated dynamical system is controllable. In computer science, fast search and mixed search methods are combined to create a fast-mixed search algorithm that is directly related to the zero forcing number  \cite{FMY}.  Being able to turn every vertex blue in a graph also models controlling the phase of electricity through an electrical network. In \cite{HHHH}, Haynes et al. researched the problem of monitoring an electric power system by placing as few measurement devices as possible and the explicit connection with zero forcing was later made in \cite{DIRST}.  For a more robust literature review see \cite{FHI} and \cite{FHII}.
 
In light of these applications, there is concern about how a faulty vertex in a network disrupts the flow of information or the ability to control a network.  A more recent variation on zero forcing was introduced to address these concerns in \cite{DK} by Dillman and Kenter called leaky forcing. In this new variation, the following question was explored. What if there is a leak in a system which prevents the zero forcing process from finishing? This paper will further explore this question, and also explore what properties make a network resistant to leaks.

Leaky forcing follows the same color change rule as zero forcing. A \emph{leak} in a graph $G$ is a vertex that is not allowed to perform a force. In particular, if a blue vertex $v$ is a leak, and it is adjacent to exactly one white vertex, then $v$ cannot force the white vertex blue. An {\it $\ell$-leaky forcing set} for a graph $G$ is a subset of initial blue vertices $B$ such that if  any $\ell$ vertices are chosen to be leaks (after $B$ has been specified), then iteratively applying the color change rule will force every vertex in $G$.  The {\it $\ell$-leaky forcing number} of $G$ is the size of a minimum $\ell$-leaky forcing set and is denoted $\Z\sub{\ell}(G).$ Notice that a $0$-leaky forcing set for a graph $G$ is also a zero forcing set, and the $0$-leaky forcing number is also the zero forcing number. This notation differs from the notation introduced by  Dillman and Kenter in \cite{DK}. We use the  term $\ell$-leaky forcing instead of the abbreviated $\ell$-forcing. This change is made in order to prevent confusion between leaky forcing and another generalization of zero forcing called $k$-forcing.

Intuitively, the more leaks there are in a graph, the larger the initial blue set must be. This is captured in the following proposition.

\begin{prop}\label{ineq}\cite{DK}
For any graph $G$, $$\Z\sub0(G)\leq \Z\sub1(G) \leq \dots \leq \Z\sub{n}(G).$$
\end{prop}

A fundamental problem concerning $\ell$-leaky forcing is determining when the inequalities in Proposition \ref{ineq} are actually equalities (or strict inequalities).  To this end, a  graph $G$ is said to be {\it $\ell$-resilient} if $\Z\sub 0(G)=\Z\sub{\ell}(G)$. 

Section \ref{theory} presents results relating $(\ell-1)$-leaky forcing sets  to $\ell$-leaky forcing sets for general graphs. Theory is then developed for some general bounds on $\ell$-leaky forcing numbers.  Section \ref{structure} explores which structural properties a graph $G$ must have for $\Z\sub0(G)$ to equal $\Z\sub{\ell}(G)$, and studies how vertex or edge deletion affect the $\ell$-leaky forcing number. The $\ell$-leaky forcing number for trees, super triangles, and new bounds for the $\ell$-leaky forcing number for grid graphs are given in Section \ref{families}.

\section{Characterization of $\ell$-leaky forcing sets}
\label{theory}

In general, the intuition behind $\ell$-leaky forcing sets is that the set of initial blue vertices is robust. If $\ell$ vertices in the graph cannot perform  a force, then the rest of the blue vertices can take over those forcing   responsibilities.
To help formalize this intuition, we will define a few concepts that let us analyze the zero forcing process. 
In general, let $B\subseteq V(G)$ be an initial set of blue vertices in $G$. 
If vertex $u$ colors $v$ blue, then we say that $u$ forces $v$ and denote it by $u\fs v$. 
The symbol $u\fs v$ is called a force. 
A \emph{set of forces $F$ of $B$ in $G$} is a set of forces such that there is a chronological ordering of the forces in $F$ where each force is valid and the whole graph turns blue. 
Intuitively, $F$ represents the instructions for how $B$ can force $G$ blue, or provides a proof that $B$ is a zero forcing set. 
Implicitly, $F$ gives rise to discrete time steps in which sets of white vertices turn blue. Given an initial blue set $B$ and forcing process $F$, we say that $B\subseteq B'\subseteq V(G)$ is \emph{obtained from $B$ using $F$} if $B$ can color $B'$ blue using only a subset of $F$.

The set $B^{[\infty]}$ is the set of blue vertices after the zero forcing rule has been exhaustively applied with $B$ as an initial blue set. Furthermore, $\binf_L$ will be determined after a set of leaks $L$ has been chosen. 
A sequence of forces $x_1\fs x_2, x_2\fs x_3,\dots, x_{k-1}\fs x_k$ is a \emph{forcing chain} and will be abbreviated by $x_1\fs x_2\fs x_3\fs\cdots\fs x_k$. A forcing chain is 
\emph{maximal} in $F$ (and implicitly given $B$) if  $x_1$ is in $B$ and $x_k$ does not perform a force in $F$. Note that $x_1=x_k$ is possible if $x_1\in B$ and $x_1$ does not perform a force in $F$.
Let $\calf (B)$ denote the set of all possible forces given a vertex set $B$. That is $u\fs v\in \calf(B)$ if there exists a set of forces $F$ of $B$ in $G$ that contains $u\fs v$.  Given this notation, $B$ is an $\ell$-leaky forcing set if for every $L\subseteq V(G)$ with $|L|=\ell$ there exists a forcing process $F$ such that if $u\fs v$, then $u\notin L$.  This idea can be used to relate an $\ell$-leaky forcing set to its set of possible forces $\calf (B)$.

\begin{lem}\label{lplus}
Let $G$ be a graph. If $B$ is an  $\ell$-leaky forcing set, then for all $v\in V(G)\setminus B$, there exists $x_1\fs v,x_2\fs v,\dots, x_{\ell+1}\fs v\in \calf(B)$ with $x_i\neq x_j$, $i\neq j$.
\end{lem}

\begin{proof}
Suppose that there exists $v\in V(G)\setminus B$ such that $v$ can be forced by at most $x_1,\dots, x_k$ distinct vertices where $k\leq \ell$. Let $L=\{x_1,\dots, x_k\}$. By construction, there does not exist a forcing process $F$ of $B$ that avoids $L$. In particular, if $F$ is a forcing process of $B$ with $x\fs v\in F$, then $x\in L$. Therefore, $B$ is not an $\ell$-leaky forcing set. 
\end{proof}

Lemma \ref{lplus} formalizes the intuition that an $\ell$-leaky forcing set is robust enough to force every vertex despite $\ell$ vertices that cannot perform forces. In fact, the set of possible forces can be used to characterize $1$-leaky forcing sets. Theorem \ref{double} captures the notion that a zero forcing set $B$ can force any white vertex in two ways if and only if $B$ is a $1$-leaky forcing set. The strategy for proving Theorem \ref{double} is to construct a forcing process that  avoids using particular vertices. This is the subject of  Lemma \ref{switch}.

Let $B$ be a fixed blue set with forcing processes $F$ and $F'$. The idea  is to use forcing process $F$ to obtain $B'$ from $B$, and then continue forcing with process $F'$. To formalize this idea, suppose $S\subseteq V(G)$ and let \[F(S)=\{ x\fs y\in F: y \notin S\}.\] By extension, \[F\setminus F(S)=\{ x\fs y\in F: y\in S\}.\] The following lemma proves that abandoning process $F$ to follow process $F'$ creates a new forcing process. 

\begin{lem}\label{switch}
Let $B$ be a zero forcing set in $G$ with zero forcing processes $F$ and $F'$. Then $(F\setminus F(B'))\cup F'(B')$ is a forcing process of $B$ for any $B'$ obtained from $B$ using $F.$
\end{lem}

\begin{proof}
Let $B'$ be some set of blue vertices obtained from $B$ using $F.$ Let $v_1<v_2<\cdots<v_n$ be an ordering of the vertices of $G$ so that whenever $v_i$ is blue for all $i<k$, then there exists $j<k$ such that $v_j\fs v_k\in F'$ which is valid. Since $F'$ is a forcing process, such an order can be found by simply keeping track of the order in which $F'$ turns vertices in $G$ blue.

For the sake of contradiction, suppose that $F'(B')$ cannot force $G$ blue given initial set $B'$. In particular, exhaustively apply forces in $F'(B')$ to $B'$ in order to  obtain $B^*\neq V(G)$. Let $k$ be the smallest index in such that $v_k$ is not in $B^*$. This implies that $v_i\in B'$ for all $i<k$. By construction of the ordering of the vertices of $G$, there exist a $j<k$ such that $v_j\fs v_k\in F'$ that can be performed if $B^*$ is blue. Since $v_k\notin B^*$, we know that $v_k\notin B'$. Therefore, $v_j\fs v_k \in F'(B')$; which contradicts the fact that all the forces in $F'(B')$ have been exhaustively applied.
\end{proof}

In essence, Lemma \ref{switch} guarantees that two forcing processes for a set $B$ can be combined. However, another interpretation is that we can switch to a new process at any point of an old process. In this way, sets of forces function as instructions that can be exchanged at any time. In the proof of Theorem \ref{double}, we use Lemma \ref{switch} to  construct a forcing process that is not obstructed by an arbitrarily fixed leak. However, the proof can also be summed up in the following diagram:

\[ B\overset{F}{\longrightarrow} B' \overset{F'}{\longrightarrow} V(G).\]

The diagram is intended to be read as: obtain $B'$ from $B$ using forcing process $F$, then obtain $V(G)$ from $B'$ using $F'$. The fact that going from $B'$ to $V(G)$ actually only uses forces in $F'(B')$ is a technical detail and will generally be suppressed. Instead, we will think of $F'$ as being a process of $B'$ in $G$.

\begin{thm}\label{double}
A set $B$ is a $1$-leaky forcing set if and only if for all $v\in V(G)\setminus B$, there exists $x\fs v,y\fs v\in \calf(B)$ with $y\neq x$.
\end{thm}

\begin{proof}
If $B$ is a $1$-leaky forcing set, then Lemma \ref{lplus} gives $x\fs v, y\fs v\in \calf(B)$ with $y\neq x$ for all $v\in V(G)\setminus B$.

Now assume for all $v\in V(G)\setminus B$, there exists $x\fs v, y\fs v\in \calf(B)$ with $y\neq x$. Let $z$ be a leak. It suffices to show that there exists a forcing process of $B$ that does not contain a force originating at $z$. Let $F$ be some set of forces of $B$ in $G$. If $z$ is the end of a forcing chain of $F$, then $B$ can force $G$ using $F$. Therefore, assume that $z$ is not the end of a forcing chain in $F$. This implies that there exists $v$ such that $z \fs v\in F$. Let $B'$ be a set of blue vertices obtained form $B$ using $F$ such that $z\fs v$ is valid given $B'$, but $v\notin B'$.  By assumption, there exists $y$ such that $y\fs v\in \mathcal F(B)$ and $y\neq z$. Therefore, there exists a set of forces $F'$ of $B$ in $G$ with $y\fs v$. Since $y\fs v\in F'$, we know that $F'$ can force $N[y]\setminus\{v\}$ without vertices in $N[v]\setminus\{y\}$. Thus, $z\fs u\notin F'(B')$ for any $u\in V(G)$. Therefore, $(F\setminus F(B'))\cup F'(B')$ is a forcing process that does not use a force originating from $z$. 
\end{proof}

A straight forward generalization of Theorem \ref{double} is false. In particular, the converse of Lemma \ref{lplus} does not hold by the following counterexample. Consider $G=K_{\ell+1}\square K_2$ where $\ell\geq 2$ with $V(G) = \{u_i:1\le i \le \ell+1\}\cup \{v_i:1\le i \le \ell+1\}$ with $u_i$'s and $v_i$'s both inducing cliques (see Figure \ref{fig:cliquematching}). Let the blue set $B=\{u_i:1\le i \le \ell+1\}$ and notice that every vertex in $\{v_i:1\le i \le \ell+1\}$ can be forced in $\ell+1$ ways.  However, setting $\{u_2,\dots,u_{\ell+1}\}$ as leaks will prevent $B$ from coloring the rest of the graph blue. Therefore, $B$ is not an $\ell$-leaky forcing set. 

 \begin{figure}[H]
        \centering{
        \includegraphics[width=0.5\textwidth]{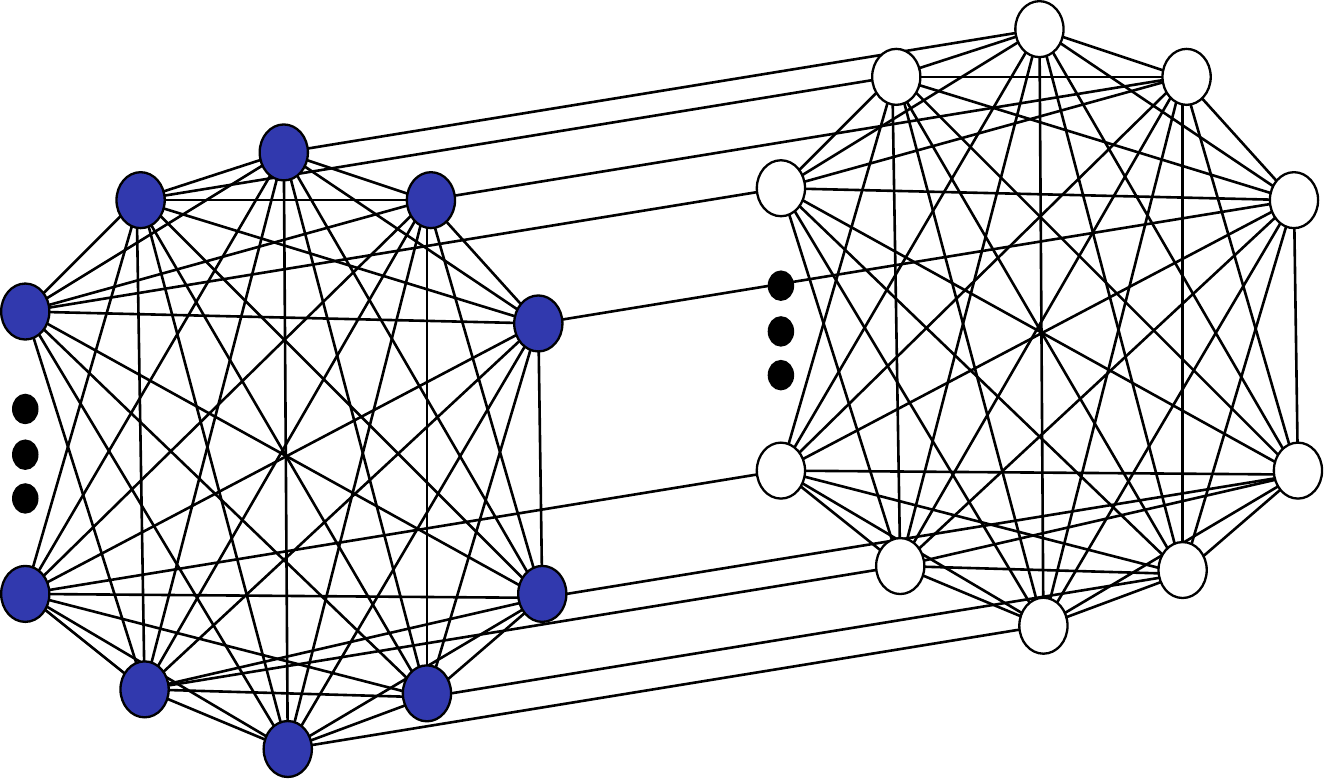}
        \caption{Counterexample to an attempted generalization of Theorem \ref{double}.}
        \label{fig:cliquematching}}
\end{figure}

Theorem \ref{setofforces} requires new definitions.  For a set of leaks $L$, let $\mathcal F_L(B)$ be the set of forces that are possible with initial blue set $B$ and leaks  $L$.  In other words, $u\fs v\in \calf_L(B)$ implies $u\notin L$ and there exists a forcing process $F$ such that $u\fs v\in F$. Before we proceed with the proof of Theorem \ref{setofforces}, we need to understand how a zero forcing set is disrupted by leaks. 

\begin{lem}\label{fail}
If $B$ is an $(\ell-1)$-leaky forcing set and $L$ is a set of $k\geq \ell$ leaks, then $|L\setminus B_L^{[\infty]}|\leq k-\ell$.
\end{lem}

\begin{proof}
Assume that $L$ is a set of $\ell$ leaks, and let $|L\setminus B^{[\infty]}_L|\geq k-\ell+1$. Every vertex in $B^{[\infty]}_L$ has either $0$, $1$, or  at least $2$ white neighbors. If $v\in B^{[\infty]}_L$ such that $v$ has one white neighbor, then $v\in L$.  Since $|L|=|L\setminus B_L^{[\infty]}|+|L\cap B^{[\infty]}_L|$, it follows that $|L\cap B^{[\infty]}_L|\leq \ell-1$. Notice that leaks in $L\setminus B^{[\infty]}_L$ did not change the zero forcing behavior of $B$. In particular, these leaks never played a role in stopping $B$ from propagating because they never were blue. Therefore, $L\cap B^{[\infty]}_L$ is a set of at most $\ell-1$ leaks which show that $B$ is not an $(\ell-1)$-leaky forcing set.
\end{proof}

Notice that if $k=\ell$, then Lemma \ref{fail} says that an $(\ell-1)$-leaky forcing set $B$ will always be able to turn any set of $\ell$ leaks blue. Given an $(\ell-1)$-leaky forcing set $B$ and a set of $\ell$ leaks $L$, we can find a time at which every leak is blue.  The proof of Theorem \ref{double} singles out this time step, and argues that there is enough freedom in our choice of forces to proceed despite the leaks. The the proof of Theorem \ref{setofforces} singles out a time when at least $\ell-1$ leaks are blue, and recognizes that the constellation of blue and white vertices at this time is very similar to the situation covered in Theorem \ref{double}. 

The following diagram roughly depicts how first part of the proof of Theorem \ref{setofforces} goes: 

\begin{equation} B\overset{F}{\longrightarrow} B'\supseteq B^* \overset{F'}{\longrightarrow} V(G).\label{diagramforces}\end{equation}

\noindent We never explicitly name $F'$ in the proof. Instead, we use Lemma \ref{fail} to find $B'$ and then whittle down $G$ into a subgraph $G^*$ such that $B^*\subseteq B'$ is a $1$-leaky forcing set of $G'$. Using this fact, we implicitly find a forcing process $F'$ of $G^*$ by invoking Theorem \ref{double}. By construction, $F'$  circumvents the only leak that is not blue in $G^*$ and can be appended to  $F$ once $B'$ is colored blue using Lemma \ref{switch}.

\begin{thm}\label{setofforces}
A set $B$ is an $\ell$-leaky forcing set if and only if $B$ is an $(\ell-1)$-leaky forcing set such that for every set of $\ell -1$ leaks $L$ and $v\in V(G)\setminus B$ there exists $x\fs v , y\fs v\in \mathcal F_L(B)$ with $y\neq x$.
\end{thm}

\begin{proof}
Let $B$ be an $(\ell-1)$-leaky forcing set such that for every set of $\ell-1$ leaks $L$ and $v\in V(G)\setminus B$ there exists $x\fs v, y\fs v\in \mathcal F_L(B)$ with $y\neq x$. Let $L'$ be a set of $\ell$ leaks. 

Since $B$ is an $(\ell-1)$-leaky forcing set, by Lemma \ref{fail} we can apply forces  until all leaks in $L'$ are blue. Let the resulting set of blue vertices be $B'$, let $L''\subseteq L'\cap B'$ be a set of $\ell-1$ leaks, and $B^*=B'\setminus L''$. Notice that if a blue vertex is a leak or has performed a force, then it can be safely deleted without altering the zero forcing behavior of the remaining blue vertices.  

Consider $G^*=G-L''$. Since $B$ is an $(\ell-1)$-leaky forcing set, $B^*$ is a zero forcing set of $G^*$.  We know that for all $v\in V(G)\setminus B$, there exists $x\fs v,y\fs v \in \calf_{L''}(B)$ in $G$. However, since vertices in $L''$ are not allowed to perform forces, it follows that $x,y$ are not in $L''$. Thus, $x\fs v, y\fs v\in \calf(B^*)$ in $G^*$, and $B^*$ is a $1$-leaky forcing set of $G^*$ by Theorem \ref{double}. Therefore, $B$ could force $G$ despite $L'$, and $B$ is an $\ell$-leaky forcing set of $G$. For a summary of this part of the proof, see diagram (\ref{diagramforces}) above.

To prove the contrapositive of the forward direction, suppose that $B$ is an $(\ell-1)$-leaky forcing set with a set of $\ell-1$ leaks $L$ and a vertex $v\in V(G)\setminus B$ such that there does not exist $x\fs v, y\fs v\in \mathcal F_L(B)$ with $y\neq x$. Since $B$ is an $(\ell-1)$-leaky forcing set, there must exist exactly one force $x\fs v\in \mathcal F_L(B)$. However, $L\cup \{x\}$ is a set of $\ell$ leaks that will prevent $B$ from coloring $v$ blue. Therefore, $B$ is not an $\ell$-leaky forcing set.
\end{proof}

Theorem  \ref{setofforces} provides a way of analyzing when $G$ is $\ell$-resilient and more generally when $\Z\sub k(G) = \Z\sub \ell(G)$. In particular, we can determine if a $k$-leaky forcing set $B$ is in fact an $\ell$-leaky forcing set by repeatedly applying Theorem \ref{setofforces}. In this way, Theorem \ref{setofforces} is a way to build up a characterization of $\ell$-leaky forcing sets and which graphs are $\ell$-resilient. To illustrate this point, notice that a zero forcing set $B$ is a $2$-leaky forcing set if and only if $B$ satisfies conditions in Theorem \ref{setofforces} for $\ell=1$ and $\ell=2$. 

Theorem \ref{double} results in Corollary \ref{2z0}, but first we introduce some definitions. Let $B$ be a zero forcing set of a graph $G$. A {\it reversal} of $B$ given a set of forces $F$ is the set of vertices of $G$ that appear at the ends of maximal forcing chains given $F$. Equivalently, let the reversal of $B$ given $F$ be denoted by \[R_F(B)=\{v\in V(G): v\fs y\notin F, \forall y\in V(G)\}.\] Often, we do not specify a particular forcing process $F$ and would be happy with any forcing process. In this case, we let $R(B)$ denote a reversal of $B$ with the understanding that there is some underlying forcing process $F$ which is unnamed and unspecified. 

In \cite{BBFHHSVV}, Barioli et al. proved the following theorem:

\begin{thm}\cite{BBFHHSVV}\label{reverse}
If $B$ is a zero forcing set of $G$, then so is any reversal of $B.$
\end{thm}

\noindent The proof of Theorem \ref{reverse} shows that a reversal $R_F(B)$ is a zero forcing set by reversing the forces in $F$. That is, if $F$ is a forcing process of $B$, then $F'=\{x\fs y: y\fs x\in F\}$ is a forcing process for $R_F(B)$.

\begin{cor}\label{2z0}
For all graphs $G$, $\Z\sub 1(G)\leq 2\Z \sub 0(G).$ In particular, if $B$ is a zero forcing set of $G$ and $R(B)$ is a reversal, then $B\cup R(B)$ is a $1$-leaky forcing set.
\end{cor}

\begin{proof}
Let $B$ be a minimum zero forcing set of $G$, and let $R(B)$ denote a reversal of $B$.  For all $v\in V(G)\setminus (B\cup R(B))$, there exists $x\neq y$ such that $x\fs v, y\fs v\in \mathcal F(B\cup R(B))$ by proof of Theorem \ref{reverse}. Therefore, by Theorem \ref{double}, $B\cup R(B)$ is a $1$-leaky forcing set of $G$. Finally, notice that $|R(B)|=|B|$ to conclude that $\Z\sub 1(G)\leq 2\Z \sub 0(G).$
\end{proof}

If equality in Corollary \ref{2z0} holds for $G$, then for every minimum zero forcing set $B$, $v\in B$, and forcing process $F$, there exists $u\in V(G)\setminus B$ such that $v\fs u\in F$.  That is, every vertex in a minimum zero forcing set actually performs a force. This follows from the fact that if $v\in B\cap R(B)$, then $v$ is the start and end of a chain in a zero forcing process. 

One of the interpretations of a leak is that $v$ is a leak if $v$ must appear at the end of its zero forcing chain. Recall that a reversal of a zero forcing set $B$ consists of the ends of the forcing chains of $B$ given a process $F$. Therefore, if we have an $\ell$-leaky forcing set $B$ and $L$ is a set of $\ell$ leaks, then $L\subseteq R(B)$. Furthermore, we can use an  $\ell$-leaky forcing set to build a zero forcing set around an arbitrary set of $\ell$ vertices by appealing to the reversal. 

\begin{lem}\label{reversal}
If $B$ is has an $\ell$-leaky forcing set of $G$, then any $\ell$ vertices of $G$ are in a zero forcing set of size $|B|$.
\end{lem}

\begin{proof}
Let $L$ be an arbitrary set of $\ell$ leaks in $G$. Since $B$ is an $\ell$-leaky forcing set, there exists a forcing process $F$ that colors $G$ blue. Let $R_F(B)$ be the reversal of $B$ given $F$. Notice that $L\subseteq R_F(B)$ since vertices in $L$ do not perform forces but eventually become blue. Furthermore, $R_F(B)$ is a zero forcing set by Theorem \ref{reverse}.
\end{proof}

The idea in the proof of Lemma \ref{reversal} is that given an $\ell$-leaky forcing set $B$, we can cleverly pick leaks to guarantee certain graph properties. For Lemma \ref{reversal}, clever leak choice lets us find zero forcing sets that contain the leaks. The next theorem shows that  clever leak choice leads to a general lower bound for $\ell$-leaky forcing sets.

\begin{thm}\label{lower}
Given a graph $G$ on $n$ vertices with $2\leq \ell\leq n-3,$ we have $\ell + 2 \le \Z_{(\ell)}(G).$
\end{thm}

\begin{proof}
For the sake of contradiction, let $B=\{x_1,\dots, x_{k}\}$ be an $\ell$-leaky forcing set with $k\leq \ell+1$. Setting $L=B$ as leaks shows that $k=\ell+1$. Furthermore, we can conclude that $x_i$ has exactly one neighbor in $V(G)\setminus B$ by setting $L=B\setminus \{x_i\}$ as leaks. Set $x_2,\dots, x_{\ell+1}$ as leaks to obtain a forcing chain $x_1\fs v_1\fs v_2\fs \dots \fs v_{n-\ell-1}.$ This implies that $v_iv_j\in E(G)$ if and only if $j=i+1$. Let $v_i$ be the neighbor of $x_2$ in $V(G)\setminus B$. If $1\leq i\leq n-\ell-2$, then setting $v_i,x_3,\dots,x_{\ell+1}$ as the leaks gives a contradiction. Therefore, $i=n-\ell-1$. Let $v_j$ be the neighbor of $x_3$ in $V(G)\setminus B$. Without loss of generality, suppose that $|j-1|\leq |n-\ell-1-j|$. Now set $L=\{v_j,x_2,x_4,\dots x_{\ell+1}\}$ as leaks. Notice that $v_{j+1}$ exists since $n-\ell-1\geq 2$. Furthermore, $v_{j+1}$ will not get forced given $L$ and $B$ is not an $\ell$-leaky forcing set. This is a contradiction.
\end{proof}

\section{Resilience and structural results}
\label{structure}

In this section, structural properties of $\ell$-resilient graphs are explored, and various differences between zero forcing sets and $\ell$-leaky forcing sets are shown. There is a natural relationship between the structure of a graph, and the behavior of $\ell$-leaky forcing sets of the graph. On one hand, the graph structure places restrictions on $\ell$-leaky forcing sets. On the other hand, certain forcing behaviors set constraints on the underlying graph.

The following lemma is an example of how graph structures place restrictions on $\ell$-leaky forcing sets.

\begin{lem}\label{degree}\cite{DK}
For any graph $G$, any $\ell$-leaky forcing set will contain at least those vertices in $G$ of degree $\ell$ or less.
\end{lem}

\noindent  Using Lemma \ref{degree}, $B$ must be chosen to be everything if $\ell=\Delta(G)$, since every vertex has degree $\Delta(G)$ or less. Proposition \ref{max} shows that if $B$ contains every vertex, then $\ell\geq \Delta(G)$.

\begin{prop}\label{max}
For any graph $G$ on $n$ vertices, $\Z\sub \ell(G)= n$ if and only if $\Delta(G) \leq \ell.$
\end{prop}

\begin{proof}
If $\Delta(G)\leq \ell$, then $\Z\sub \ell(G) = n$ by Lemma \ref{degree}. 

Let $v$ be a vertex with $d(v) \geq \ell+1$. Let $B= V(G)\setminus\{v\}$ be the blue set, and let $L$ be a set of $\ell$ leaks. By the pigeon hole principle,  there is a vertex $u\in N(v)\setminus L$. Therefore, $u$ can force $v$. Thus, $B$ is an $\ell$-leaky forcing set and $\Z\sub \ell (G)<n$.
\end{proof}

A bound for the minimum degree of an $\ell$-resilient graph can be given using the fact that a vertex with $\ell$ neighbors must be blue, and placing $\ell$ leaks does not prevent the graph from forcing. 

\begin{prop}\label{minimum}
If $G$ is $\ell$-resilient, then $\delta(G)\geq \ell+1.$
\end{prop}

\begin{proof}
Let $B$ be a minimum $\ell$-leaky forcing set of $G$ and assume that there exists a vertex $v$ with $\ell$ or fewer neighbors. By Lemma \ref{degree}, $v\in B$. 

If $v$ does not perform a force, set $L=N(v)$ as leaks. Since $G$ is $\ell$-resilient, $B$ forces $V(G)$ despite $L$.  Therefore $B\setminus \{v\}$ is a smaller zero forcing set for $G$, a contradiction. 

If $v$ forces $u$, set $L=N[v]\setminus \{u\}$ as leaks. By Lemma \ref{lplus}, there exists at least one other vertex $w$ not in $N[v]$ that forces $u$ blue. Therefore, $B\setminus\{v\}$ is a smaller zero forcing set, which is also contradiction. 
\end{proof}

One way of reading Proposition \ref{minimum} is that if $G$ is $\ell$-resilient, then $G$ cannot contain a small edge-cut surrounding a single vertex. Rather than considering  leaks that isolate just a single vertex, we can consider leaks that separate whole portions of a graph. In this sense, Theorem \ref{kab} follows the spirit of Proposition \ref{minimum}, and shows that if $G$ is $\ell$-resilient, then $G$ does not contain a small or dense edge-cut.

\begin{thm}\label{kab}
If $G$ is $\ell$-resilient, then $G$ does not have an edge-cut $C$ such that the edges of $C$ induce a subgraph of $K_{a,b}$ with $a+b\leq \ell$.
\end{thm}

\begin{proof}
Let $B$ be a minimum $\ell$-leaky forcing set, and let $C$ be a minimum edge-cut such that the edges of $C$ induce a subgraph of $K_{a,b}.$ This implies that $|V(C)|\leq \ell$. Since $C$ is a minimum edge-cut, there exists exactly two components $G_1,G_2$ in $G-C$. Let $L=V(C)$ be a set of leaks. Since $B$ is an $\ell$-leaky forcing set, there exists forcing process $F$ of $B$ such that $B$ forces $V(G)$ despite $L$. Furthermore, $B'=B\cap V(G_1)$ is a zero forcing set of $G_1$ with forcing process $F'=F(V(G_2))$. Therefore, $R_{F'}(B')$ is a zero forcing set of $G_1$ that contains $V(C)\cap V(G_1)$ by Lemma \ref{reversal}.  Let $x\in V(C)\cap V(G_1)$. Now $(B\cap V(G_2))\cup (R_{F'}(B')\setminus\{x\})$ is a zero forcing set of $G$.
Since $|R_{F'}(B')|=|B'|=|B\cap V(G_1)|$,
\[|(B\cap V(G_2))\cup (R_{F'}(B')\setminus\{x\})|=|(B\cap V(G_2))|+|B\cap V(G_1)|-1=|B|-1.\]
Thus, $G$ is not $\ell$-resilient.
\end{proof}

Notice that a $K_{\ell+2}$ is an $\ell$-resilient graph with an edge-cut that induces a subgraph of $K_{1,\ell+1}$. Therefore, Theorem \ref{kab} can potentially be improved.

In \cite{BBFHHSVV}, the authors considered whether there is a graph $G$ with a vertex $v$ such that $v$ is in every zero forcing set of $G$. It was shown that no such graph with non-trivial components  exists.

\begin{thm}\cite{BBFHHSVV}
If $G$ is a connected graph of order greater than one, then no vertex is in every minimum zero forcing set of $G.$
\end{thm}

\noindent Notice that by changing the question to reflect $\ell$-leaky forcing, the answer is affirmative by either Lemma \ref{degree} or Proposition \ref{max}. 

Furthermore, \cite{BBFHHSVV} considered whether there is a graph with a unique minimum zero forcing set.  Since the reversal of a minimum zero forcing set is also a minimum zero forcing set by Theorem \ref{reverse}, no such non-empty graph exists. 

\begin{cor}\cite{BBFHHSVV}
    If $G$ is a connected graph of order greater than one, then $G$ does not have a unique minimum zero forcing set.
\end{cor}

\noindent However, there do exist graphs with unique minimum $\ell$-leaky forcing sets with $\ell\geq 1$. To see this take the complete graph $K_{\ell+2}$ with $\ell$ leaves at every vertex. Let $B$ be the set of leaves. By Lemma \ref{degree}, $B$ has to be in every  minimum $\ell$-leaky forcing set, and $B$ is an $\ell$-leaky forcing set. See Figure \ref{fig:unique} for an example when $\ell=1$.
 
  \begin{figure}[H]
        \centering{
        \includegraphics[width=0.2\textwidth]{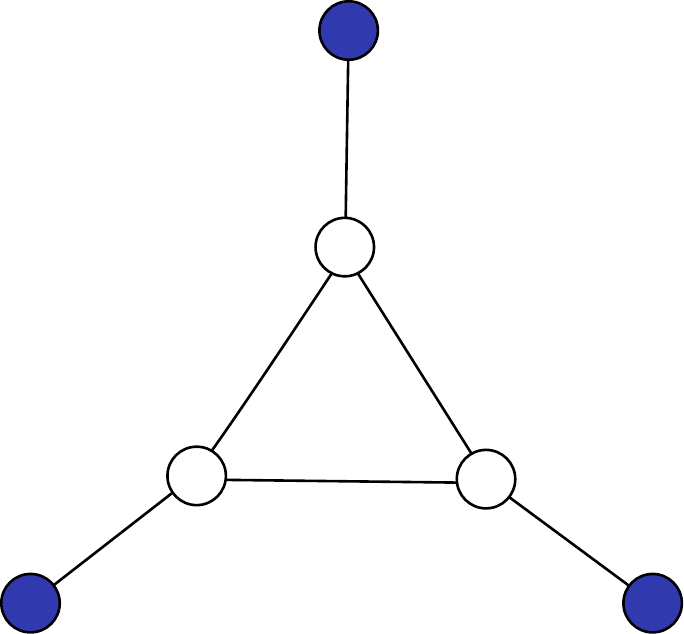}
        \caption{An example of a graph with a unique $1$-leaky forcing set.}
        \label{fig:unique}}
    \end{figure}

 The following theorem from \cite{EHHLR} shows how removing an edge from a graph impacts the zero forcing number.
 
\begin{thm}\label{123}\cite{EHHLR}
For every graph $G$ and every edge $e$ in $G$, 
$$-1\leq \Z(G)-\Z(G-e)\leq 1.$$
\end{thm}

Proposition \ref{delete} investigates how edge-deletion affects the $1$-leaky forcing number.

\begin{thm}\label{delete}
For every graph $G$ and every edge $e$ in $G$,

$$-2 \le \Z\sub1(G) - \Z\sub1(G-e).$$

\end{thm}

\begin{proof} Let $B$ be a minimum $1$-leaky forcing set and $e = vw$.

\textbf{Case 1:} Assume both $v,w\in B$. Notice that $\calf (B)$ is the same given $G$ and $G-e$, since $v\fs w$ and $w\fs v$ are not in $\calf(B)$ given $G$.  Therefore, $B$ is a $1$-leaky forcing set for $G-e$ by Theorem \ref{double}. Thus, $\Z\sub1(G-e)\leq \Z\sub1(G)$.

\textbf{Case 2:} Without loss of generality, assume $v\in B$ and $w\notin B$. Therefore, there exists $x_1,y_1 \in V(G)$ such that $x_1\fs w, y_1\fs w \in \mathcal{F}(B)$ with $x_1\neq y_1$ by Theorem \ref{double}. There are two subcases to consider.

\textbf{Subcase 1:} Assume $v\notin \{x_1,y_1\}.$ Then $B$ is a $1$-leaky forcing set for $G-e$ since every vertex in $V(G-e)\setminus B$ can still be forced in two ways. Thus, $\Z\sub1(G-e)\leq \Z\sub1(G)$. 

\textbf{Subcase 2:} Without loss of generality, assume $v=x_1.$ If $F$ is a forcing process of $B$ in $G$, then $F' = F\setminus \{v\fs w\}$ is a forcing process of $B\cup \{w\}$ in $G-e$. Then $B'=B\cup \{w\}$ is a $1$-leaky forcing set for $G-e$. Thus, $\Z\sub1(G-e)\leq \Z\sub1(G)+1$.

\textbf{Case 3:} Assume $v,w \notin B$. Therefore, there exists $x_1,y_1,x_2,y_2 \in V(G)$ such that $x_1\fs w, y_1\fs w,x_2\fs v, y_2\fs v \in \mathcal{F}(B)$ with $x_1\neq y_1$ and $x_2\neq y_2$ by Theorem \ref{double}. There are three subcases to consider.

\textbf{Subcase 1:} Assume $v\notin \{x_1,y_1\}$ and $w\notin \{x_2,y_2\}.$ Then $B$ is a $1$-leaky forcing set for $G-e$ since $v$ and $w$ (or any other white vertex) can be forced in two ways without $e$. Thus, $\Z\sub1(G-e)\leq \Z\sub1(G)$. 

\textbf{Subcase 2:} Without loss of generality, assume $v=x_1$ and $w\notin \{x_2,y_2\}.$ By deleting $e$, $w$ can only be forced in one way using $B$. Therefore, $B'=B\cup \{w\}$ is a $1$-leaky forcing set for $G-e$ by a similar argument to subcase $2$ of case $2$. Thus, $\Z\sub1(G-e)\leq \Z\sub1(G)+1$. 

\textbf{Subcase 3:} Without loss of generality, assume $v=x_1$ and $w=x_2$. By deleting $e$, both $w$ and $v$ might only be forced in one way. If $F$ is a forcing process of $B$ in $G$, then $F' = F\setminus \{v\fs w, w\fs v\}$ is a forcing process of $B'=B\cup \{v,w\}$ in $G-e$.
Therefore, $B'$ is a $1$-leaky forcing set for $G-e$. Thus, $\Z\sub1(G-e)\leq \Z\sub1(G)+2$. 
\end{proof}

 Figure \ref{fig:delete} shows that the inequalities in the proof of Theorem \ref{delete} can be tight. 

 \begin{figure}[H]
        \centering{
        \includegraphics[width=0.5\textwidth]{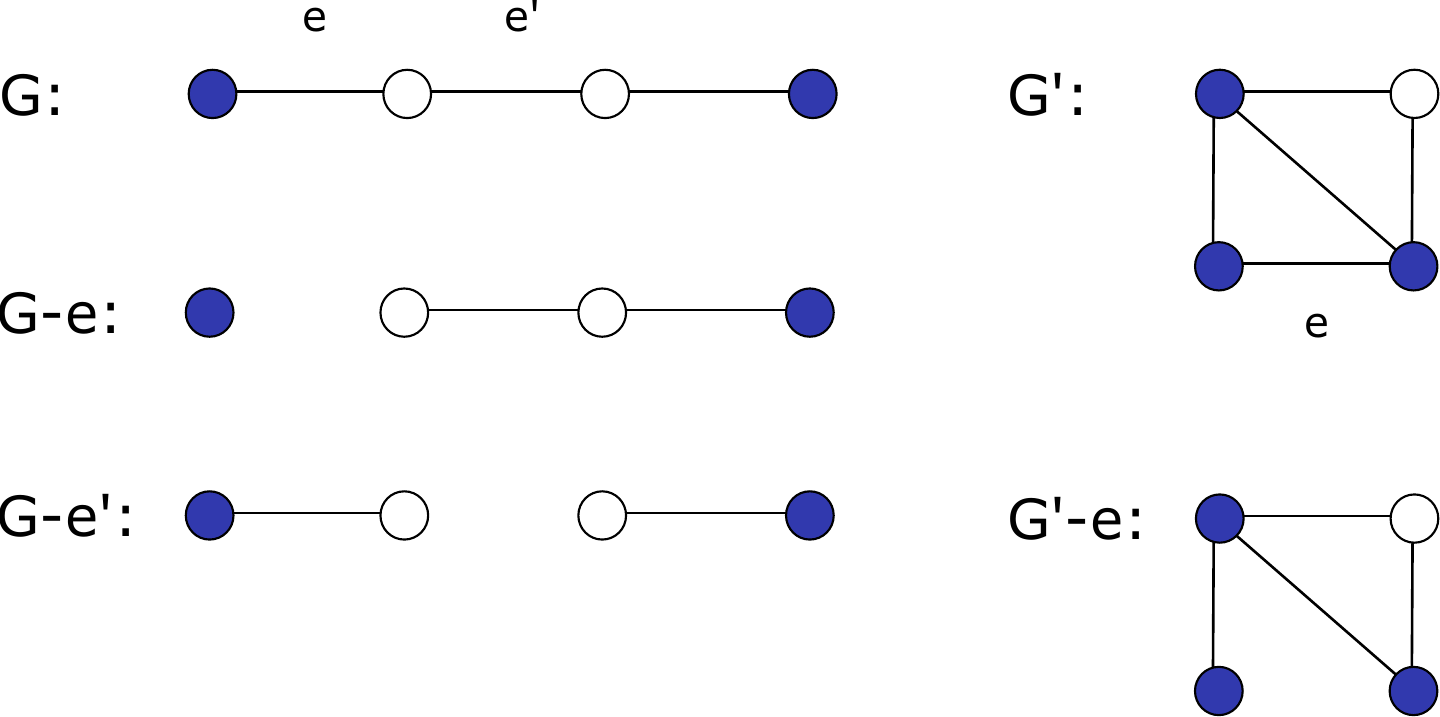}

        \caption{Examples where deleting an edge increases the $1$-leaky number by zero ($G'$ to $G'-e$), one ($G$ to $G-e$), and two ($G$ to $G-e'$).
        \label{fig:delete}}}
\end{figure}
    
The examples in Figure \ref{fig4} give some indication that $\Z\sub1(G) - \Z\sub1(G-e)$ is bounded above by $2$.

 \begin{figure}[H]
        \centering{
        \includegraphics[width=0.45\textwidth]{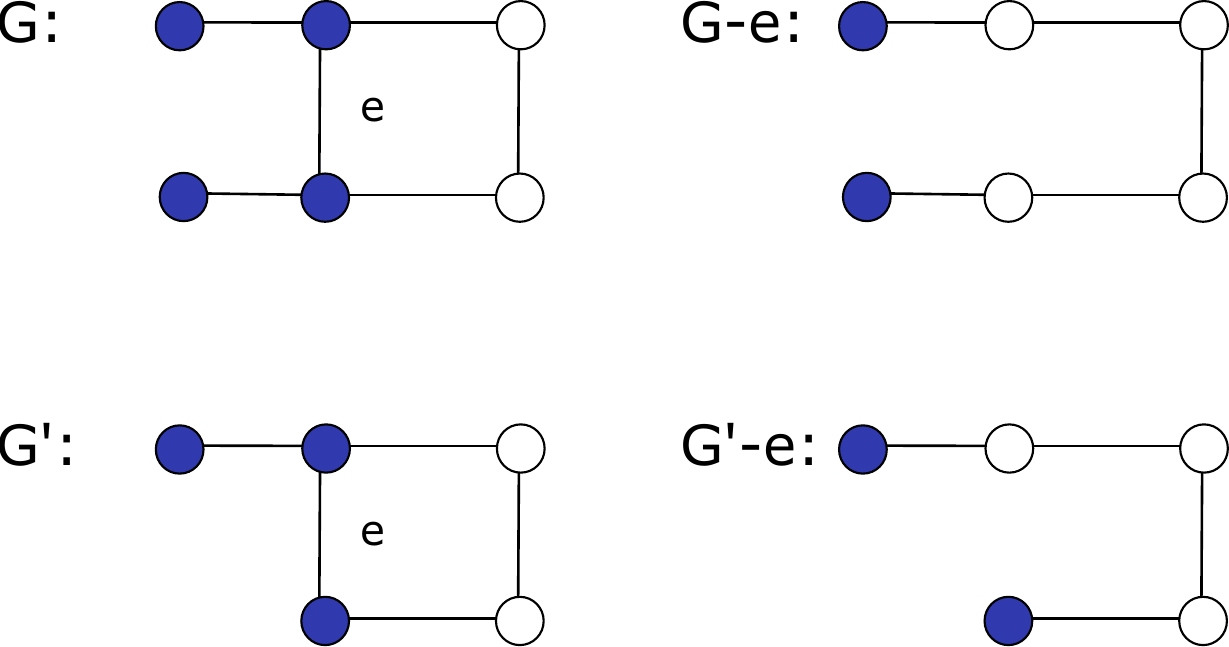}

        \caption{Examples where deleting an edge decreases the $1$-leaky forcing number by two ($G$ to $G-e$), and one ($G'$ to $G'-e$).
        \label{fig4}}}
\end{figure}

Furthermore, it has been shown the behavior of vertex deletion can be used to extract some information about zero forcing sets.

\begin{thm}\label{deletevertex}\cite{EHHLR}
Let $G$ be a graph and $v\in V(G)$. If $\Z(G)-\Z(G-v)=-1,$ then $v\notin B$ for all minimum zero forcing sets $B$ of $G$.
\end{thm}

Recall that by Lemma \ref{reversal}, if $G$ is an $\ell$-resilient graph, any $\ell$ vertices are in a minimum zero forcing set.
The following corollary, relates the behavior of vertex deletion with respect to the zero forcing number and $\ell$-resilience. Furthermore, Corollary \ref{minusone} gives a simple way of determining if a graph is not $\ell$-resilient.

\begin{cor}\label{minusone}
If $G$ is $\ell$-resilient, then $\Z(G)-\Z(G-v)\neq -1$ for all $v\in V(G).$ Equivalently, if $\Z(G)-\Z(G-v)= -1$ for some $v\in V(G)$, then $G$ is not $\ell$-resilient.
\end{cor}

\begin{proof}
Let $G$ be an $\ell$-resilient graph. By Proposition \ref{ineq}, $G$ is a $1$-resilient graph and by Lemma \ref{reversal}, every vertex is in some minimum zero forcing set. Hence, by the contrapositive of Theorem \ref{deletevertex}, $\Z(G)-\Z(G-v)\neq -1$ for all $v\in V(G).$
\end{proof}

The following example shows that unlike edge deletion, given a vertex $v$ in a graph $G$, $\Z\sub 1(G)-\Z\sub 1(G-v)$ cannot be bounded below. Consider the graph $G$ constructed by identifying an endpoint of $k$ paths of order  $n\geq 3$. Call this identified vertex $v$. Since $G$ is a tree with $k$ leaves, $\Z\sub 1(G)=k$. Notice that $G-v$ consists of $k$ disconnected paths of length $n-1$. Therefore, $\Z\sub 1(G-v)= 2k$ and $\Z\sub 1(G)-\Z\sub 1(G-v)=-k$. Recall that $k$ is arbitrary, showing that $\Z\sub 1(G)-\Z\sub1 (G-v)$ is unbounded.

\section{The $\ell$-leaky forcing number for certain graph families}
\label{families}

 Section \ref{families} applies theory developed in Sections \ref{theory} and \ref{structure} to resolve the $\ell$-leaky forcing numbers for various graph families. Recall that Section \ref{theory} ended with Theorem \ref{lower} which states that $\ell+2\leq \Z\sub \ell(G) $ as long as $2\leq \ell\leq |V(G)|-3$. Proposition \ref{prop:independent} applies this bound to restrict the order of a $(k-1)$-resilient graph with zero forcing number $k$.  

\begin{prop}\label{prop:independent}
If $\Z\sub0(G)=\Z\sub {k-1}(G)=k$ and $|V(G)|\geq 2$, then $G$ is the complete graph $K_{k+1}$, $K_a\cup\overline{K}_{k+1-a}$ with $a\geq 2$, or  $\overline{K}_{k}$.
\end{prop}

\begin{proof}
If $G=\overline{K}_k$, then we are done. Therefore, let $G$ be a graph on $n$ vertices with at least one edge. By Theorem \ref{lower}, $n-2\leq k-1$, implying that $n\leq k+1$. Furthermore, if $n\leq k$, then $\Z\sub 0(G)\leq k-1$. Thus, $n=k+1$. By Proposition 2.2 in \cite{R2012}, $K_{k+1}$ is the only connected graph on $k+1$ vertices with zero forcing number equal to $k$. Therefore, we will assume that $G$ is disconnected. Notice that $G$ cannot have more than one  nontrivial component without contradicting the zero forcing number. Let $C$ be the non-trivial component of $G$. If $\Z\sub0(C)< |V(C)|-1$, then we have a contradiction. Therefore, $C$ is a clique by Proposition 2.2 in \cite{R2012} and the proof is complete.
\end{proof}

The proof idea from Theorem \ref{lower} can also be applied for the case where $\ell=1$. However, here we find that it is possible for $\Z\sub 1(G)\leq 2$. In particular, we can combine the idea of clever leak placement with Proposition \ref{minimum} to determine which graphs have $1$-leaky forcing number equal to $2$. Note that Dillman and Kenter found that paths and cycles have $1$-leaky forcing number equal to $2$ \cite{DK}. 

\begin{thm}\label{cycle}
Let $G$ be a graph on $n\geq 2$ vertices. Then, $\Z\sub1(G)=2$ if and only if $G=C_n$ with $\Z\sub 0(G)=2$ or $G=P_n$ with $\Z\sub 0(G)=1$.
\end{thm}

\begin{proof}
The backward direction of the statement is proven in \cite{DK}. 

 Assume that $\Z\sub0(G)=\Z\sub1(G)=2$. Let $\{x,y\}$ be a $1$-leaky forcing set of $G$. Since $\Z\sub 0(G)=\Z\sub 1(G)$, we can conclude that  $\delta(G)\geq 2$ by Proposition \ref{minimum}. Therefore, $xy\in E(G)$ and $d(x),d(y)=2$. Set $L=\{y\}$ as a leak to obtain a forcing chain $x\fs v_1\fs v_2\fs \dots \fs v_{n-2}.$ This implies that $v_iv_j\in E(G)$ if and only if $j=i+1$. Note that $y$ is adjacent to exactly one vertex $v_i$ with $1\leq i \leq n-2$. If $i<n-2$, then set $L=\{v_i\}$ as a leak. In this case, $\{x,y\}$ cannot force $v_{n-2}$, which is a contradiction. Therefore, $i=n-2$ and $G$ is a cycle.
 
If $\Z\sub0(G)=1$, then $G$ is a path by a result in  \cite{R2012}. 
\end{proof}

By considering all the cases where $\Z\sub 1(G)=2$ and noting that $\Z\sub 1(G)>1$, we can conclude that the general lower bound in Theorem \ref{lower} extends to $\ell=0,1$ provided that we rule out paths and cycles. 

\begin{cor}
If $G$ is a graph of order $n$, $G$ is not $P_n$ or $C_n$, and $0\leq \ell \leq n-3$, then $\ell + 2 \le \Z_{(\ell)}(G)$.
\end{cor}

Lemma \ref{degree} states that any $\ell$-leaky forcing set must contain all vertices of degree $\ell$ or less. It is interesting to consider under which conditions \[B=\{ v\in V(G): d(v)\leq \ell\}\] is an $\ell$-leaky forcing set of $G$. Though we do not have a complete answer to this question, Theorem \ref{tree} confirms that trees belong to this class of graphs. Note that the case where $\ell =1 $ of Theorem \ref{tree}
was proven in \cite{DK}.

\begin{thm}\label{tree}
If $T$ is a tree and  $B$ is the set of vertices in $T$ with degree at most $\ell$ where $\ell\geq1$, then $\Z\sub{\ell}(T)=|B|.$
\end{thm}

\begin{proof}
Let $L$ be a set of $\ell$ leaks and note that Lemma \ref{degree} implies $\Z\sub\ell(T)\geq |B|$.  For the sake of contradiction, suppose that $B^{[\infty]}_L\neq V(T)$. Obtain the graph $H$ by deleting all blue vertices in $\binf_L$ with at most one white neighbor from $T.$ Notice that $H$ is a forest. Let $T'$ be a component of $H$. 
 
 Suppose that $T'$ is an isolated vertex, $v$. Then $v$ has at least $\ell+1$ blue neighbors in $T$, each of which has $v$ as its only white neighbor. Therefore, $N_T(v)\subseteq L$ which is a contradiction. Thus, we may assume that $T'$ is a tree on at least two vertices. 
 
 Notice that all of the leaves of $T'$ are white. Furthermore, $T'$ has at least two leaves $u$ and $v$. Therefore, $u$ has at least $\ell$ blue neighbors in $T$ each of which has $u$ as its only white neighbor. The same holds for $v$. It follows that $(N_T(u)\cup N_T(v))\cap B^{[\infty]}_L\subseteq L$, which is a contradiction.  Thus, $B^{[\infty]}_L= V(T)$ and $B$ is an $\ell$-leaky forcing set.
\end{proof}

The {\it $n^{th}$ supertriangle,} $T_n,$ is an equilateral triangular grid with $n$ vertices on each side.
Supertriangles are another example that the implied lower bound in Lemma \ref{degree} can be tight. In particular, Theorem \ref{triangle} shows that if $\ell =4,5$, then the vertices of a $T_n$ with degree less than or equal to $\ell$ do form an $\ell$-leaky forcing set. Notice that the supertriangle has three sides and a natural orientation. Let the sides be given by $S_0,R_0,M_0$ where $S_0$ is the left side, $R_0$ is the right side, and $M_0$ is the bottom side in Figure \ref{fig:super}. Furthermore, let $S_i$ be the left side of $T_n-\bigcup_{j<i}S_{j}$. In this sense, we can consider the left diagonal ``rows" of $T_n$ indexed from the left. Define $R_i$ and $M_i$ similarly. Furthermore, use the depiction of the supertriangle in Figure \ref{fig:super} to partition the edges of $T_n$ into three sets: the set  of horizontal edges $H$,  the set of edges pointing from the upper left corner of the page to the lower right corner of the page $SE$, and the edges pointing from the lower left corner of the page to the upper right corner of the page $NE$.

 \begin{figure}[H]
    \centering{
    \includegraphics[width=0.4\textwidth]{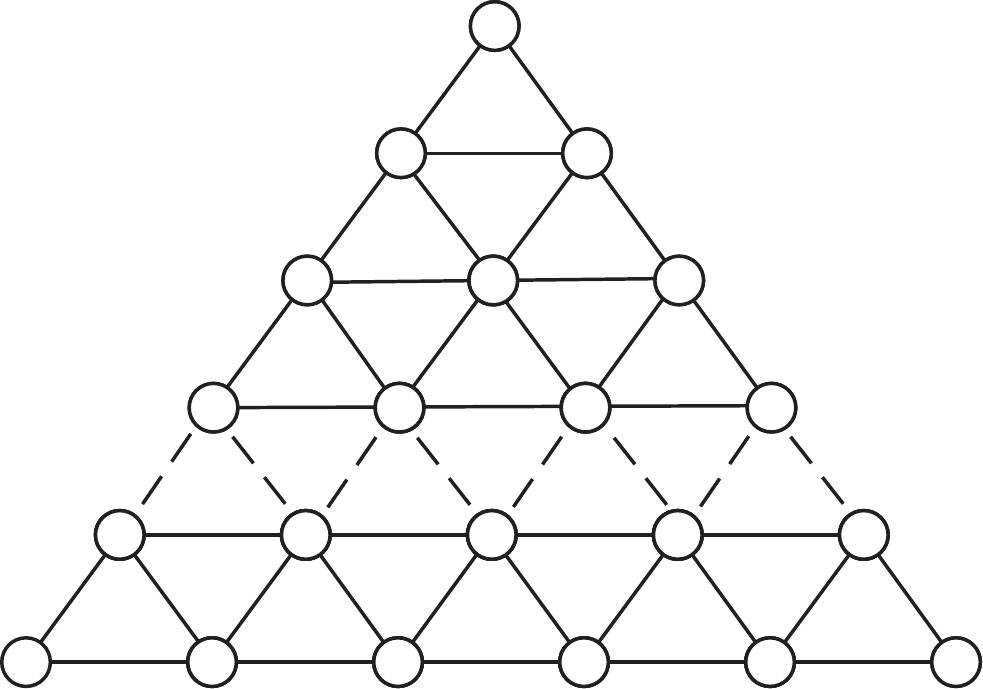}
    \caption{The supertriangle, $T_n$.}
    \label{fig:super}}
\end{figure}

\begin{thm}\label{triangle}
For $T_n,$ 
\begin{enumerate}
    \item $\Z\sub 0(T_n)=\Z\sub 1(T_n)=n$,
    \item $\Z\sub 2(T_n)\leq \Z\sub 3(T_n)\leq 2n-1$, and
    \item  $\Z\sub 4(T_n)=\Z\sub 5(T_n)=3n-3$.
\end{enumerate}
\end{thm}

\begin{proof}
1. In \cite{AIM}, it was shown that $\Z\sub0(T_n)=n$ by taking $S_0$, $M_0$, or $R_0$ as a blue set of $T_n$. We will show that $\Z\sub 1(T_n)=n$ by finding a zero forcing set of size $n$ that forces every vertex in two ways. Let $B=M_0$. Notice that vertices in $M_1$ can be forced by two distinct vertices in $M_0$ along $NE$ edges or along $SE$ edges. That is, every vertex in $M_1$ can be forced in two ways. By way of induction, vertices in $M_{i+1}$ can be forced by two distinct vertices in $M_{i}$  along $NE$ edges or  along $SE$ edges. Therefore, $B$ is a $1$-leaky forcing set by Theorem \ref{double}. Recall that $\Z\sub0(T_n)\leq \Z\sub 1(T_n)$ by Proposition \ref{ineq}. Thus, $\Z\sub1(T_n)=n$.

2. Let $B=S_0\cup R_0$. Clearly, $B$ is a $1$-leaky forcing set. Let $x\in V(T_n)$ be a leak. Notice that $x\in S_i\cap R_j$ for some $i$ and $j$. The initial blue set $B$ can color vertices  $B'=\bigcup_{\substack{y\leq i\\ z\leq j}} S_y\cup R_z$ blue in two ways without being inhibited by the leak $x$. Notice that $B'$ is a zero forcing set of $T_n$ such that $x$ does not have any white neighbors and every remaining white vertex can be forced in two ways. Therefore, by Theorem \ref{setofforces}, $B$ is a $2$-leaky forcing set and $\Z\sub 2(T_n)\leq 2n-1.$

Now let $x_1,x_2\in V(T_n)$ be a pair of distinct leaks.  Let $i$ be the smallest index such that $S_i$ contains a leak. Similarly, let $j$ be the smallest index such that $R_j$ contains a leak. The initial blue set $B$ can color vertices $B'=\bigcup_{\substack{y\leq i\\ z\leq j}} S_y\cup R_z$ blue in two ways without being inhibited by $x_1,x_2$. There are two cases: either $S_i\cup R_j$ contains both $x_1$ and $x_2$, or $S_i\cup R_j$ contains only $x_1$ (without loss of generality). 

\textbf{Case 1:} Assume that $S_i\cup R_j$ contains both $x_1$ and $x_2$. Let $S_i\cap R_j =\{x\}$. Notice that $S_i\setminus \{x\}$ and $R_j\setminus \{x\}$ each contain at most one leak. Therefore, $S_i$ can force the remaining white vertices using only edges in $SE$ and $H$ (where forces using edges in $H$ go from left to right). Similarly, $R_j$ can force the remaining white vertices using only edges in $NE$ and $H$ (where forces using edges in $H$ go from right to left). Thus, $B$ can force every vertex in $V(T_n)\setminus B$ in two ways despite leaks  $x_1$ and $x_2.$ 

\textbf{Case 2:} Assume that $S_i\cup R_j$ contains only $x_1$ (without loss of generality). Therefore, $S_i\cap R_j =\{x_1\}$ and $N(x_1)$ is blue. Since $B'$ is a $2$-leaky forcing set of $T_n$, we know that every vertex $v\in V(T_n)\setminus B'$ can be forced in two ways by Theorem \ref{setofforces} despite the leak $x_2$. 

In either case, $B$ can force all white vertices in two ways despite leaks $x_1$ and $x_2$. Therefore, $B$ is a $3$-leaky forcing set by Theorem \ref{setofforces}. Thus, $\Z\sub2(T_n) \le \Z\sub3(T_n) \le 2n-1$.

3.  Since $S_0$, $R_0$, and $M_0$ contain all the vertices of $T_n$ with degree $4$ or less, we know that $3n - 3 \le \Z\sub 4(T_n)$ by Lemma \ref{degree}. We will proceed by showing that $B=S_0\cup R_0\cup M_0$ is a $4$-leaky forcing set and then show $B$ is a $5$-leaky forcing set. 

Clearly, $B$ is a $3$-leaky forcing set. Let $L$ be a set of three leaks. Let $i,j,k$ be the smallest indices such that  each of $S_i, R_j, M_k$ contain leaks. Notice that $B$ can force the vertices in \[B'=\bigcup_{\substack{x\leq i\\ y\leq j\\ z\leq k}}S_x\cup R_y\cup M_z\] in two ways. There are two cases:  $S_i\cup R_j\cup M_k$ contains exactly three leaks, or $S_i\cup R_j\cup M_k$ contains exactly two leaks. In any case,  $S_i\setminus (R_j\cup M_k)$ contains at most one leak and $R_j\setminus (S_i\cup M_k)$ contains at most one leak (without loss of generality). 

\textbf{Case 1.1:} If $S_i\cup R_j\cup M_k$ contains three leaks, then the remaining vertices can be forced in two ways despite the leaks in $L$ as discussed in case 1 in the argument for claim 2.

\textbf{Case 2.1:} If $S_i\cup R_j\cup M_k$ contains two leaks, then Let $i'>i$ and $j'>j$ be the smallest indices such that $S_{i'}$ and $R_{j'}$ contain a leak. As we have seen in case 2 of the argument for claim 2, $S_i$ and $R_j$ can force the vertices in $B''=\bigcup_{\substack{i<y\leq i'\\ j< z\leq j'}}S_y\cup R_z$ in two ways despite the leaks in $L$. Furthermore, $B'\cup B''$ will force the remaining white vertices in two ways despite leaks in $L$. In either case, $B$ is a $4$-leaky forcing set by Theorem \ref{setofforces}.

Let $L$ be a set of four leaks. Let $i,j,k$ be the smallest indices such that $S_i, R_j, M_k$ contain leaks. Notice that $B$ can force the vertices \[B'=\bigcup_{\substack{x\leq i\\ y\leq j\\ z\leq k}}S_x\cup R_y\cup M_z\] in two ways.  There are three cases: $S_i\cup R_j\cup M_k$ contains either two, three, or four leaks. In any case, without loss of generality, $S_i\setminus(R_j\cup M_k)$ and $R_j\setminus(S_i\cup M_k)$ contain at most one leak each. 

\textbf{Case 1.2:} If $S_i\cup R_j\cup M_k$ contains four leaks, then $S_i$ and $R_j$ can force every remaining white vertex despite the leaks in $L$ in two ways as discussed in case 1 in the argument of claim 2. 

\textbf{Case 2.2:} Assume that $S_i\cup R_j\cup M_k$ contains three leaks. Following the argument in case 2.1 of claim 3, we see that $B'$ can force the remaining white vertices in two ways despite the leaks in $L$. 

\textbf{Case 3.2:} Assume that $S_i\cup R_j\cup M_k$ contains two leaks. Let $i'>i$, $j'>j$, and $k'>k$ be the smallest indices such that $S_{i'}$, $R_{j'}$, and $M_{k'}$ contain a leak. In this is situation, $S_{i'}\cup R_{j'}\cup M_{k'}$ necessarily contains two leaks. Let \[B''=\bigcup_{\substack{i<x\leq i'\\ j< y\leq j'\\ k<z\leq k'}}S_x\cup R_y\cup M_z\] and notice that $B'$ can force each vertex in  $B''$  in two ways despite the leaks in $L$. Finally, $B'\cup B''$ can force the remaining white vertices as discussed in case 1 in the argument for claim 2.

In any case, $B$ can force each white vertex in two ways despite the leaks in $L$. Therefore, $B$ is a $5$-leaky forcing set by Theorem \ref{setofforces}. Thus, $\Z\sub 4(T_n)=\Z\sub 5(T_n)=3n-3$.
\end{proof}
 
 A {\it rectangular grid graph} is the Cartesian product $P_n \square P_m$. Let $V(P_n \square P_m)=\{(x,y): 1\leq x\leq n, 1\leq y\leq m\}$ and $(x,y)(a,b)\in E(G)$ if and only if $x=a$ and $y\pm 1=b$ or $y=b$ and $x\pm 1=a$. We say that $(x,y)$ is the vertex in row $x$ and column $y$. The motivation for studying these graphs is due to a question posed by Dillman and Kenter \cite{DK}. However, the grid graph has also been studied extensively in more general zero forcing contexts \cite{AIM}.

\begin{quest}\cite{DK}\label{quest:DK}
If $1\le n \le m$, is there an $n\ge 10$, such that $\Z\sub1(P_n \square P_m) > \min\{2n,m\}$?
\end{quest}

  In pursuit of answering Question \ref{quest:DK}, Dillman and Kenter established several upper bounds on $\Z\sub 1(P_n\square P_m)$.  The upper bound arguments in \cite{DK} are based on describing a set of initially colored vertices and doing case analysis to show the set is a $1$-leaky forcing set, e.g. the proof of Theorem \ref{thm:DK4.20} uses this technique.  A generalization of the $1$-leaky forcing set used in Theorem \ref{thm:DK4.20} is used to establish Theorem \ref{thm:pnpm} by applying Theorem \ref{double}.  This shows that the answer Question \ref{quest:DK} is ``No." 
 
 \begin{thm}\label{thm:DK4.20}\cite{DK} If $1 \le n \le m$ and $n \le \left\lfloor \dfrac{m}{2} \right\rfloor + 2$, then $\Z\sub 1(P_n\square P_m) \le m$.
 \end{thm}

  Theorem \ref{thm:DK4.20}, Propositions \ref{prop:DK4.12}, and \ref{prop:DK4.18} allow the focus of the proof of Theorem \ref{thm:pnpm} to be on ``large" values of $n$.  Note that the bound in Proposition \ref{prop:DK4.12} can also be achieved using Corollary \ref{2z0}.
 
\begin{prop}\cite{DK}\label{prop:DK4.12}
If $1\le n \le m$, then  $\Z\sub 1(P_n\square P_m) \le 2n$.

\end{prop}

\begin{prop}\label{prop:DK4.18}\cite{DK}
If $1\leq n$, then $\Z\sub 1 (P_n \square P_n) = n.$
\end{prop}

The argument for Theorem \ref{thm:pnpm} describes a zero forcing set $B$ with $m$ vertices such that each white vertex can be forced in two different ways.  The next two definitions help describe a set of forcing chains of $B$.  Define the set of \emph{boundary} vertices of $P_n\square P_m$, denoted $\mathcal{B}_0$, as the set of vertices of degree three or less.  Define the set of \emph{$y$-interior} vertices of $P_n\square P_m$, denoted $\mathcal{B}_y$, as the set of vertices whose distance to a nearest boundary vertex is $y$ for $1\le y \le \left\lceil n/2\right\rceil -1$.  In addition, define $\mathcal{C}_y = (P_n\square P_m)[\mathcal{B}_y]$.  Note that $\mathcal{C}_y$ is an induced cycle except for the largest value of $y$. See Figure \ref{cy} for an example of a decomposition of $P_n\square P_m$ into $\mathcal C_y$. We say that a graph $H$ contains  a forcing chain $C$ if $V(C)\subseteq V(H)$ and $x\fs y\in C$ implies that $xy\in E(H).$  Similarly, we say a forcing chain $C$ contains a graph $H$ if $V(H)\subseteq V(C)$ and $xy\in E(H)$ implies $x\fs y\in C$. The purpose of this language is to describe the relative layout of forcing chains in a graph. 

\begin{figure}[H]
    \centering{
  \includegraphics[width=0.75\textwidth]{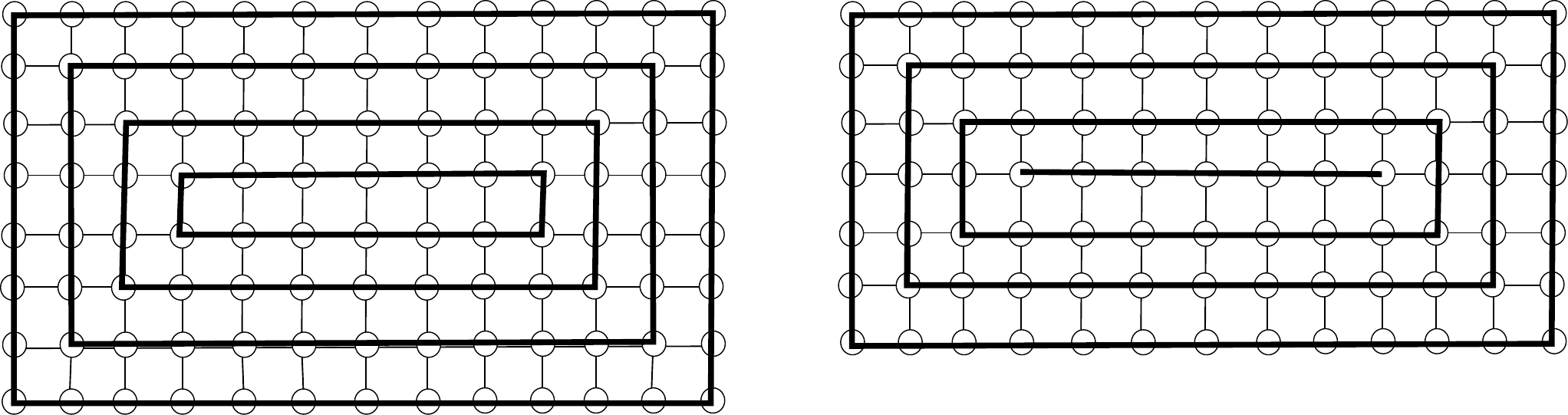}
  \caption{Decompositions of grid graphs into $\mathcal C_y$ when $n$ is even (left image) and when $n$ is odd (right image).}
    \label{cy}}
\end{figure}

\begin{thm}\label{thm:pnpm}
If $1 \le n \le m$, then $\Z\sub 1(P_n\square P_m) \le \min\{2n,m\}$.
\end{thm}

\begin{proof} If $n=m$, then $\Z\sub 1 (P_n\square P_m) = n$, by Proposition \ref{prop:DK4.18}, so it will be assumed that $n < m$.  If $n\in \{1,2,3,4,5\}$, Theorem \ref{thm:DK4.20} and Proposition \ref{prop:DK4.12} give the desired result.  This allows the focus of the remainder of the proof to be the cases when $6\le n < m$.  


Let $x$ be the smallest integer such that $n \le 2x+3 \le m$.  Since $n+1 \le m$, such an integer will always exist.  Define $k = \left\lfloor \dfrac{m-2x}{2} \right\rfloor$, and $j = m - 2x - k$.  Note that $j\ge 1$  and $k\ge 2$ by construction.  Let \[B_1 = \{(i,x+k),(x+1,x+p) \, : \, 1\le i \le x, \,1\le p \le k\},\]  \[B_2 = \{(i,x+k+1),(x+1,x+k+p)\, : \, 1\le i \le x, 1\le p \le j \},\] and $B = B_1 \cup B_2$. Note that $|B| = m$.

\begin{figure}[H]
    \centering{
  \includegraphics[width=0.85\textwidth]{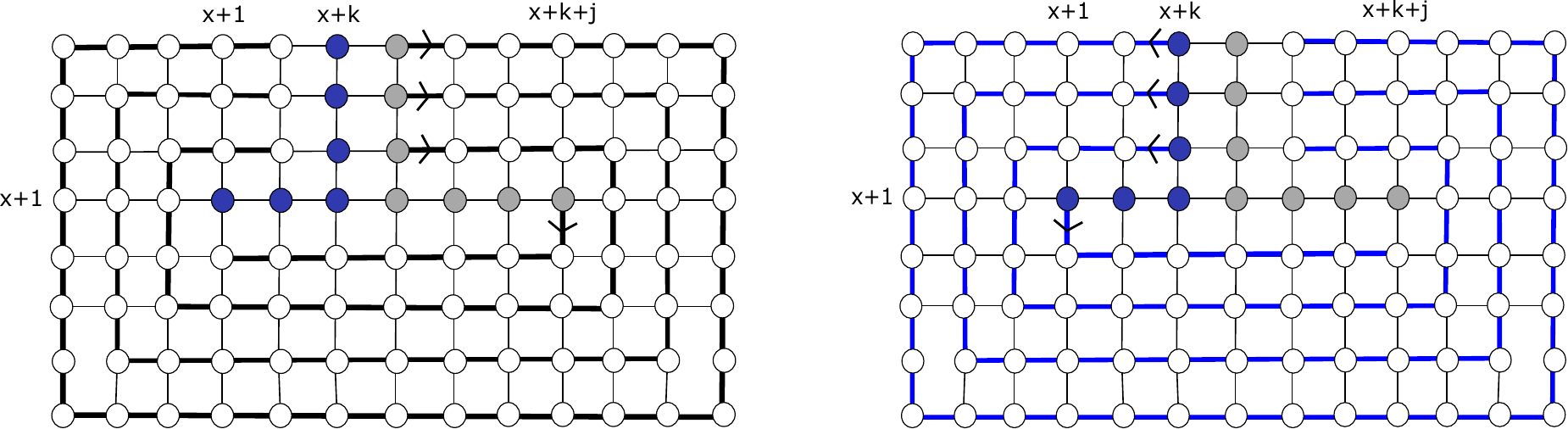}
  \caption{Forcing chains originating from $B_2$ (grey vertices, left image) and from $B_1$ (dark vertices, right image) in $P_8\square P_{13}$.}
    \label{gridpattern}}
\end{figure}

The argument will show that $B_1$ or $B_2$ can each force the entire graph without vertices in $B_2$ or $B_1$, respectively, performing a force.  More importantly, the description of the forcing chains shows that every vertex in $V(P_n\square P_m)\setminus B$ can be forced in two different ways. In particular, we will find two disjoint forcing processes given in Figure \ref{gridpattern}.

\textbf{Case 1:} Assume that $n$ is even.
If vertices in $B_1$ do not perform forces, then there is a forcing process that creates forcing chains originating with vertices in $B_2$. These forcing chains can be found in $\mathcal{C}_y$.  More particularly, the forcing chain in $\mathcal{C}_{i-1}$ includes $(i,x+k+1), (i,x+k+2)$ and $(i,x+k-1)$ and is oriented in a clockwise direction (see Figure ??) for $1\le i \le x$, and the final forcing chain in $\mathcal{C}_{n/2 - 1}$ includes $(x+1,x+k+j), (x+2,x+k+j)$, and $(x+2,x+1)$.  If  vertices in $B_2$ do not perform forces, then the forcing chains originating with vertices in $B_1$ include $(i,x+k)$, $(i,x+k-1)$, and $(i,x+k+2)$ and are oriented in a counter-clockwise direction for $1\le i \le x$, and the final forcing chain includes $(x+1,x+1), (x+2,x+1)$ and $(x+2,x+j+k)$.

\textbf{Case 2:} Assume that $n$ is odd.
Define $p = \left\lceil n/2 \right\rceil$.  If  vertices in $B_1$ or $B_2$ do not perform forces, then the description of the forcing chains that reside in $\mathcal{C}_y$ for $0 \le y \le p -3$ are the same as the first case.  If vertices in $B_1$ do not perform forces, then the forcing chain originating from vertices in $B_2$ that is in $\mathcal{C}_{p-2}$ includes $(x+1,x+k+j), (x+3,x+k+j)$, $(x+3,x+1)$, and $(x+2,x+1)$ and is oriented in a clockwise direction.  The final forcing chain is not in $\mathcal{C}_{p-1}$, but contains $\mathcal{C}_{p-1}$.  In particular, the final forcing chain starts at $(x+1,x+k+j-1)$, which is in $B_1$ since $j\ge 2$ when $n$ is odd. The forcing chain resumes so that it includes $(x+2,x+k+j-1)$ and continues to $(x+2,x+2)$.  If vertices in $B_2$ do not perform forces, then the forcing chains we just described are augmented and reversed as in case 1.

The two cases establish that for each $v\in V(P_n\square P_m)\setminus B$, there exists $x\to v$,$y\to v \in\mathcal{F}(B)$ with $y\neq x$.  Thus, Theorem \ref{double} gives that $B$ is a $1$-leaky forcing set and $\Z\sub1(P_n\square P_m) \le m$.  Combining this with Proposition \ref{prop:DK4.12} cited earlier from \cite{DK} establishes that $\Z\sub1(P_n\square P_m) \le \min\{2n,m\}$. \end{proof}

As remarked in \cite{DK}, it is believed that $\Z\sub1(P_n\square P_m) = \min\{2n,m\}$.  Theorem \ref{thm:pnpmequality} gives some evidence that this is true.

\begin{thm}\label{thm:pnpmequality}
If $G=P_n\square P_m$ where $m\geq 2n^2$, then $\Z\sub 1 (G)=2n$.
\end{thm}

\begin{proof}
Recall that the columns of $G$ are indexed by $[m]$ and the rows of $G$ are indexed by $[n].$ Let $G[i,j]$ be the subgraph of $G$ induced by the vertices in columns $i$ through $j$.

Assume that  $B\subseteq V(G)$ is a  set of blue vertices such that $|B|\leq 2n-1$. By the pigeon hole principle, there exists $k$ such that $M=G[(k-1)n+1,kn]$ does not contain a blue vertex. Notice that if $k=1$ or $kn=m$, then setting $(1,n+1)$ or $(1,m-n)$ as a leak, respectively, shows that $B$ is not a $1$-leaky forcing set. Therefore, we will assume that $k\neq 1$ and $kn\neq m$.

Let $S= G[1,(k-1)n]$ and $R=G[kn+1,m]$. Without loss of generality, assume that $|V(S)\cap B|\leq n-1$. Consequently, $|V(R)\cap B|\geq n.$ Notice that $V(S)\cap B$ is not a zero forcing set of $S$, since $\Z\sub 0(S)=n$.  Exhaustively apply the zero forcing rule until the only remaining valid forces are $x\fs y$ with $y\in M$. Notice that at this point in the process $G[(k-1)n-1,(k-1)n]$ must contain a white vertex, otherwise, $V(S)\cap B$ could force all of $S$. This implies that there exists a vertex  $(i,(k-1)n)$ that cannot perform a force either because it is white or because it has a white neighbor in $G[(k-1)n-1,(k-1)n]$. Since $(i,(k-1)n)$ has a white neighbor that $V(S)\cap B$ cannot force, $(i,(k-1)n)$ essentially acts like a leak (it will never perform a force without help from a forcing chain originating in $V(R)\cap B$). Without loss of generality, assume that $i\leq n/2$. Therefore, $V(S)\cap B$ cannot color all of $M$ blue. Furthermore, notice that any leak in column $kn+1$ shows that $V(R)\cap B$ cannot color $M$ blue. In particular, assume that $(n+1-i,kn+1)$ is a leak. 

Let $B'_S$ be the set of blue vertices in $M$ obtained by allowing $V(S)\cap B$ to force until failure. Similarly, let $B'_R$ be the set of blue vertices in $M$ obtained by allowing $V(R)\cap B$ to  force until failure given a leak at $(n+1-i, kn+1).$ Recall that neither $V(S)\cap B$ nor $V(R)\cap B$ can color all of the vertices in $M$ blue. Therefore, $B$ is a $1$-leaky forcing set only if there exists $u\in B'_S$ and $v\in B'_R$ such that $u$ is adjacent to $v$ in $G$. 

\begin{figure}[H]
    \centering{
  \includegraphics[width=0.55\textwidth]{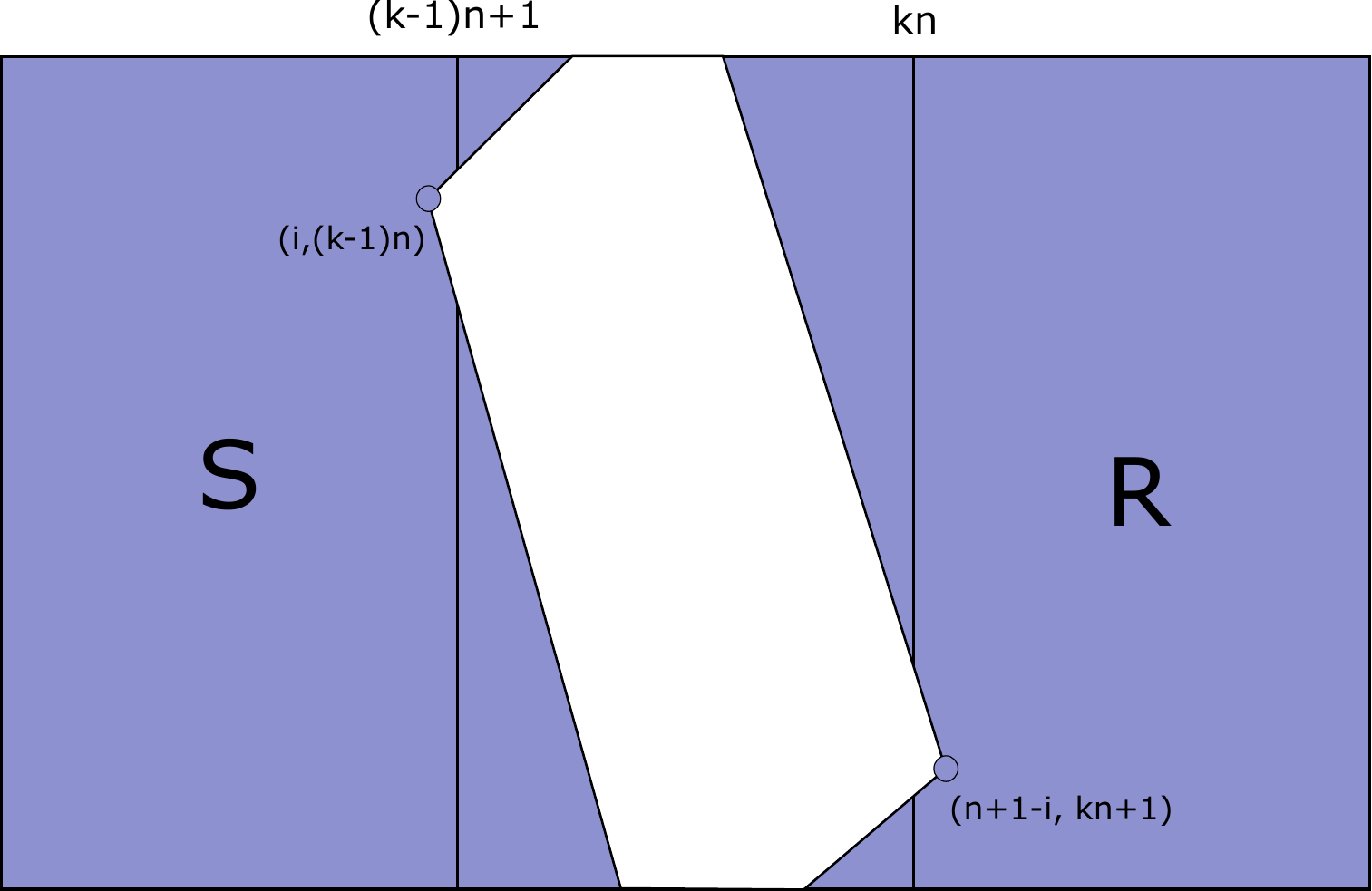}
  \caption{Relative layout of blue vertices and white vertices (not to scale) after forces have been exhaustively applied in the proof of Theorem \ref{thm:pnpmequality}. }
    \label{shadow}}
\end{figure}

Suppose that $u\in B'_S$ and $v\in B'_R$. There are three cases: (1) $u$ and $v$ are in the same row, (2) $u$ and $v$ are in the same column, or (3) $u$ and $v$ are not adjacent. We will show that (1) or (2) implies (3). For some intuition, consult Figure \ref{shadow}.

\textbf{Case 1:} Suppose that $u$ and $v$ are in row $j$. In particular, let $u=(j,u')$ and $v=(j,v')$. Since $(i, (k-1)n)$ cannot perform any forces given $V(S)\cap B$, we know that $u'< (k-1)n+|i-j|$. Similarly, the leak at $(n+1-i, kn+1)$ implies that $kn-|n+1-i-j|<v'$. 

However, 
\[|n+1-i-j|+|i-j|\leq n\] 
implies that 
\[(k-1)n+|i-j|\leq kn-|n+1-i-j|.\]
By applying the bounds on $u'$ and $v'$, we get 
\[u'<(k-1)n+|i-j|\leq kn-|n+1-i-j|<v'.\]
Therefore, $u$ and $v$ are not adjacent. 

\textbf{Case 2:} Suppose that $u$ and $v$ are in column $j$. In particular, let $u=(u',j)$ and $v=(v',j)$. Since $(i, (k-1)n)$ cannot perform any forces given $V(S)\cap B$, we know that \[u'\in [n]\setminus [i-j+1, i+j-1].\] Since $u'$ exists, we know that $j\leq n-i$. Similarly, the leak at $(n+1-i, kn+1)$ implies that \[v'\in [n]\setminus [n+1-i-(n-j),n+1-i+n-j].\] Since $v'$ exists, we know that $j\geq i+1$. Therefore, $u'\in [n]\setminus [1,i+j-1]=[i+j,n]$ and $v'\in [n]\setminus [1-i+j,n]=[1,j-i].$ Notice that $j-i$ and $j+i$ differ by at least $2$. Therefore, $u$ and $v$ are not adjacent.

Since there does not exist $u\in B'_S$ that is adjacent to $v\in B'_R$, $B$ is not a $1$-leaky forcing set of $G$. Thus, $2n\leq \Z\sub1(G)$. The corresponding upper bound is obtained by Theorem \ref{thm:pnpm}.
\end{proof}

\section{Closing remarks}

One proof technique we tried to develop is extending minimum zero forcing sets to minimum $1$-leaky forcing sets. However, this is not always possible. Consider $P_4\square P_5$  and let $B$ be a side of length $4$. This set is a minimum zero forcing set for this graph. The $1$-leaky forcing number for this graph is $5$, but there is no way to add one blue vertex to $B$ to get a minimum $1$-leaky forcing set. Therefore, there is a graph with a minimum zero forcing set that cannot be extended to a minimum $1$-leaky forcing set.

 \begin{figure}[H]
        \centering{
        \includegraphics[width=0.3\textwidth]{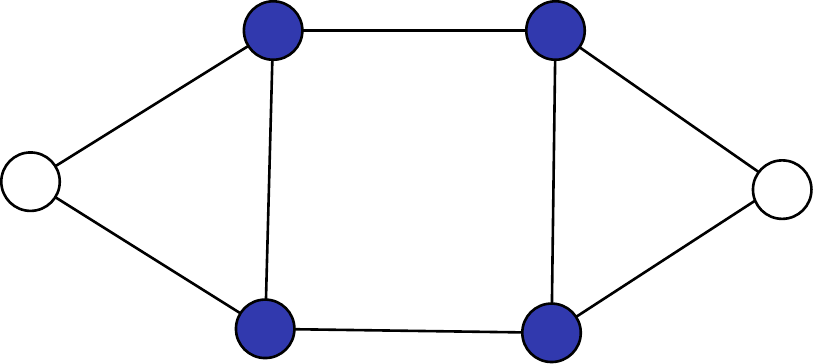}
        \caption{Graph $G$ with a minimum $1$-leaky forcing set that does not contain a minimum zero forcing set.}
        \label{fig:nominimum}}
    \end{figure}

Similarly, minimum $1$-leaky forcing sets cannot always be reduced to a minimum zero forcing set. Consider the graph $G$ depicted  in Figure \ref{fig:nominimum}. The zero forcing number for $G$ is two, and the blue set shown above is a minimum $1$-leaky forcing set. One can check that removing any two vertices from this blue set will not result in a minimum zero forcing set. Therefore, there exists a graph with a minimum $1$-leaky forcing set that does not contain a minimum zero forcing set. 

Notice that if the top row of vertices are colored blue in $G$ in Figure \ref{fig:nominimum}, there is a way to remove two blue vertices to yield a minimum zero forcing set. With these examples in mind, the following conjecture is proposed.

\begin{prob}
For every graph $G$, does there exists a minimum $\ell$-leaky forcing set $B$ such that $B$ contains a minimum zero forcing set?
\end{prob}

\section*{Acknowledgements}
This material is based upon work supported by the National Science Foundation under Grant Numbers DMS-1839918 and DMS-1719841.

\end{document}